\documentclass[12pt]{amsart}
\usepackage{amsmath}
\usepackage[active]{srcltx}
\usepackage{t1enc}
\usepackage[latin2]{inputenc}
\usepackage{verbatim}
\usepackage{amsmath,amsfonts,amssymb,amsthm}
\usepackage[mathcal]{eucal}
\usepackage{enumerate}
\usepackage[centertags]{amsmath}
\usepackage{graphics}
\numberwithin{equation}{section}

\newtheorem{theorem}{Theorem}

\newtheorem{OldTheorem}{Theorem}

\newtheorem{lemma}{Lemma}

\def\supp{{\rm supp\,}}
\def\spec{{\rm spec\,}}
\def\Im{{\rm Im\,}}
\def\Re{{\rm Re\,}}

\def\ZR{\ensuremath{\mathbb R}}
\def\ZZ{\ensuremath{\mathbb Z}}

\def\ZN{\ensuremath{\mathbb N}}
\def\ZT{\ensuremath{\mathbb T}}
\def\ZI{\ensuremath{\mathbb I}}

\begin{document}
\title[On the divergence of double Fourier series]{On the divergence of triangular and eccentrical spherical sums of double Fourier series}

\author
{G. A. Karagulyan}

\address{G. A. Karagulyan, Institute of Mathematics of Armenian National
Academy of Science, Baghramian ave. 24/5, 375019, Yerevan, Armenia}
\email{g.karagulyan@yahoo.com}

\date{}
\maketitle
\begin{abstract} We construct a continuous function on the torus with almost everywhere divergence triangular sums of double Fourier series. An analogous theorem we  also prove for eccentrical spherical sums.
\end{abstract}
\maketitle
\section{Introduction}
Carleson \cite{Car} proved that the Fourier series of any function from $L^2(\ZT)$ converges almost everywhere. Hunt \cite{Hunt}, Sj\"{o}lin \cite{Sjo1} and Antonov \cite{Ant1} established the same property of Fourier series in wider function classes. Now the best known result, due to Antonov \cite{Ant1}, proves the a.e. convergence of Fourier series for the functions from $ L\log L\log\log\log L(\ZT)$.

The problem of almost everywhere convergence of multiple Fourier series is well investigated for different definitions of partial sums. If $f\in L^1(\ZT^2)$ is an arbitrary function with the double Fourier series
\begin{equation}\label{v1}
\sum_{n,m=-\infty}^{+\infty}c_{nm}e^{i(nx+my)}
\end{equation}
and $G\subset \ZR^2$ is a bounded region, then we denote by
\begin{equation}\label{v2}
S_ G(x,y,f)=\sum_{(n,m)\in G}c_{nm}e^{i(nx+my)}.
\end{equation}
the partial sum of (\ref{v1}) over the region $G$.  Let $P\subset \ZR^2$ be an arbitrary polygon containing the origin. We set
 \begin{equation*}
 \lambda P=\{(\lambda x,\lambda y):\, (x,y)\in P\},\quad \lambda>0.
 \end{equation*}
   C.~Fefferman \cite{Ch1} proved that if $f\in L^p(\ZT^2)$, $p>1$, then
 \begin{equation}\label{v3}
 S_{\lambda P}(x,y,f)\to f(x,y)\text { a.e. as }   \lambda\to\infty.
 \end{equation}
In the case when $P$ is either rectangle or square is considered by Sj\"{o}lin  \cite{Sjo1} and Antonov \cite{Ant1}. In the rectangle case the relation (\ref{v3}) holds for any $f\in L(\log L)^3 \log\log L$ (\cite{Sjo1}). While $P$ is a square, then it holds whenever $f\in L(\log L)^2 \log\log L$(\cite{Ant1}). Tevzadze \cite{Tev} showed that for any sequence of rectangles $R_1\subset R_2\subset R_3\subset \ldots$  $\ZR^2$ with the sides parallel to the coordinate axes the partial sums $S_{R_k}(x,y,f)$ of any function $f\in L^2(\ZT^2)$  converge a.e..

Note that in all these convergence theorems the partial sums depend on one parameter. The following theorem due to C.~Fefferman \cite{Ch2} shows that the rectangular partial sums
 \begin{equation*}
 S_{NM}(x,y,f)=\sum_{|n|\le N,|m|\le M}c_{nm}e^{i(nx+my)},
 \end{equation*}
 with two independent parameters $N$ and $M$ have a quite different property.
 \begin{OldTheorem}[C.~Fefferman] There exists a real continuous function $f\in C(\ZT^2)$ such that
 \begin{equation*}
\limsup_{N,M\to\infty}|S_{NM}(x,y,f)|=\infty
 \end{equation*}
for any $(x,y)\in \ZT^2$.
 \end{OldTheorem}
 Observe that in the above discussed convergence theorems of double Fourier series the summation regions are polygons with fixed side directions. The present paper shows that a little freedom of the side directions of the summation polygons changes the situation basically.

We consider the following rhombus regions
 \begin{equation}\label{v4}
 \Delta(a,b)=\left\{(x,y)\in \ZR^2 :\,a|x| + b|y| \le1\right\},\quad a,b>0.
 \end{equation}
Given such a region $\Delta=\Delta(a,b)$, we denote
 \begin{equation*}
 \rho(\Delta)=\frac{\max\{a,b\}}{\min\{a,b\}}.
 \end{equation*}
  It is clear that ${\Delta}$ is a square, while $\rho(\Delta)=1$. Note that the regions
\begin{equation}\label{v5}
\Delta(a,b)\cap \ZR_+^2,
\end{equation}
are triangles with a vertex at the origin. It is clear that the double series (\ref{v1}) of any real function $f\in L(\ZT^2)$ can be written in the real form (by sine  and cosine  functions), and the sum (\ref{v2}), corresponding to the rhombus region  (\ref{v4}), coincides with the partial sum of the Fourier series in the real form over the triangle (\ref{v5}).

A sequence of regions  $G_k$ is said to be complete, if $\cup_{k=1}^\infty G_k=\ZR^2$.

The next theorem is an equivalent reformulation of the theorem of C.~Fefferman \cite{Ch1}.
\begin{OldTheorem}[C. Fefferman]\label{OT1}  If $\Delta_k$, $k=1,2,\ldots$, is a complete increasing sequence of squares of the form (\ref{v4} ($\rho(\Delta_k)=1$), then for any function  $f\in L^p (\ZT^2)$, $p>1$, the relation
\begin{equation*}
\lim_{k\to\infty} S_{\Delta_k}(x,y,f)=f(x,y)
 \end{equation*}
 holds almost everywhere.
\end{OldTheorem}
In the present paper we prove the following theorem, which shows that in Theorem \ref{OT1} the condition $\rho(\Delta_k)=1$ can not be replaced by $\rho(\Delta_k)\to 1$.
\begin{theorem}\label{T1}
There exists a real continuous function $f\in C(\ZT^2)$ and a complete sequence of regions $\Delta_k$, $k=1,2,\ldots,$ of the form (\ref{v4} such that $\Delta_{k}\subset \Delta_{k+1}$, $\rho(\Delta_k)\to 1$ and
\begin{equation}\label{v6}
\limsup_{k\to\infty} |S_{\Delta_k}(x,y,f)|=\infty
 \end{equation}
almost everywhere.
\end{theorem}
An example of a function $f\in L^p(\ZT^2)$, $1\le p<\infty$, satisfying the same relation (\ref{v6}) was constructed in \cite{Kar2}. An analogous divergence theorem for Walsh-Fourier series was considered in \cite{Kar3}.

We obtain also a similar divergence theorem for some spherical sums. Let $B=B(x_0,y_0,r)$ be the open ball, with the radius $r$ and the center at the point $(x_ 0,y_0)$. We define the following quantity
\begin{equation*}
\tau(B)=\frac{\sqrt {x_0^2+y_0^2}}{r},
\end{equation*}
describing the eccentricity of the ball against the origin.  We prove
\begin{theorem}\label{T2}
There exists a continuous function $f\in C(\ZT^2)$ and a complete sequence of balls $U_k$, $k=1,2,\ldots,$ such that $\tau(U_k)\to 0$ and
\begin{equation}\label{v7}
\limsup_{k\to\infty} |S_{U_k}(x,y,f)|=\infty
 \end{equation}
almost everywhere.
\end{theorem}
In the proofs of the theorems we use the method applied in the paper \cite{Kar1},  where we establish the unboundedness of the maximal directional Hilbert transform on the plane, associated with an arbitrary infinite family of directions.

Unfortunately, we are not able to prove Theorem \ref{T2} with the condition $\tau(U_k)=0$ instead of $\tau(U_k)\to 0$. That would be a negative answer to the well known problem on almost everywhere convergence of spherical partial sums of double Fourier series.

\section{Auxiliary lemmas }
Let $\ZT=\ZR/(2\pi \ZZ)$ be the one dimensional torus and $\ZT^2=\ZT\times \ZT$. If $E$ is a Lebesgue measurable set in $\ZT$ or $\ZT^2$, then the notation $\ZI_E$ stands for the indicator function of $E$, $|E|$ denotes the Lebesgue measure of $E$. For any $n\in\ZN$ and for a measurable set $E\subset \ZT^2$ we set
\begin{equation*}
E(n)=\{(x,y)\in\ZT^2:\, (nx,ny)\in E\}.
\end{equation*}
It is clear that $|E(n)|=|E|$. The relation
\begin{equation}\label{u1}
\lim_{n\to\infty}|A\cap B(n)|=\frac{|A||B|}{4\pi^2},\quad (4\pi^2=|\ZT^2|),
\end{equation}
is well known and follows from a theorem of Fe\`{j}er (see for example \cite{Zyg}, Theorem (4.15)).  The following two lemmas are based on a standard probabilistic independence argument.

\begin{lemma}\label{L21}
Let $n_0>0$ be an arbitrary integer and $0<\alpha<1$. Then for any sequence of measurable sets $E_k\subset \ZT^2$, $k=1,2,\ldots, l$,  with $|E_k|>4\pi^2\alpha$, there exist natural numbers  $n_0<n_1<n_2<\ldots<n_l$ satisfying the condition
\begin{equation}\label{u2}
\left|\bigcup_{k=1}^lE_k(n_k)\right|>4\pi^2(1-(1-\alpha)^l).
\end{equation}
\end{lemma}
\begin{proof}
From (\ref{u1}) it follows that
\begin{align*}
\lim_{n\to\infty}|(A\cup B(n))^c|&=\lim_{n\to\infty}\left(4\pi^2-|A|-|B(n)|+|A\cap B(n)|\right)\\
&=4\pi^2-|A|-|B|+\frac{|A||B|}{4\pi^2}=\frac{(4\pi^2-|A|)(4\pi^2-|B|)}{4\pi^2}\\
&=\frac{|A^c|\cdot|B^c|}{4\pi^2}.
\end{align*}
Taking small enough $\delta>0$,  then applying this relation successively $l-1$ time, we may find integers  $1=n_1<n_2<\ldots<n_l$ such that
\begin{align*}
\left|\left(E_1\cup E_2(n_2)\cup\ldots\cup E_l(n_l)\right)^c\right|&<\frac{|E_1^c|\cdot|E_2^c|\cdot \ldots\cdot|E_l^c|}{(4\pi^2)^{l-1}}+\delta\\
&< \frac{(4\pi^2-4\pi^2\alpha)^l}{(4\pi^2)^{l-1}}=4\pi^2(1-\alpha)^l.
\end{align*}
This  immediately gives (\ref{u2}).
\end{proof}
\begin{lemma}\label{L22}
Let $E_k\subset \ZT^2$ be a sequence of measurable sets such that  $|E_k|>4\pi^2\alpha$, $k=1,2,\ldots $, where  $0<\alpha<1$. Then there exists an infinite  sequence of integers $0<n_1<n_2<\ldots $ such that
\begin{equation}\label{u3}
\left|\bigcap_{l\ge 1}\bigcup_{k\ge l}E_k(n_k)\right|=4\pi^2=|\ZT^2|.
\end{equation}
\end{lemma}
\begin{proof}
Applying Lemma \ref{L21} to the set families
\begin{equation*}
\left\{E_j:\, k^2<j\le (k+1)^2\right\},\quad k=0,1,2,\ldots,
\end{equation*}
we find integers $n_j$ satisfying
\begin{equation*}
\left|\bigcup_{j=k^2+1}^{(k+1)^2}E_j(n_j)\right|>4\pi^2(1-(1-\alpha)^{2k+1}),\quad k=0,1,\ldots.
\end{equation*}
Thus we get $|\cup_{k\ge l}E_k(n_k)|=4\pi^2$ for any $l=1,2,\ldots$, and so (\ref{u3}).
\end{proof}

\begin{lemma}\label{L}
For any function $f\in L^\infty(\ZT^2)$ we have
\begin{equation}\label{u4}
|\{(x,y)\in\ZT^2:\, |f(x,y)|>\|f\|_1/8\pi^2\}|\ge  \frac{4\pi^2}{8\pi^2(\|f\|_\infty/\|f\|_1)-1}.
\end{equation}
\end{lemma}
\begin{proof}
Denote
\begin{equation*}
E=\{(x,y)\in\ZT^2:\, |f(x,y)|>\|f\|_1/8\pi^2\}.
\end{equation*}
We have
\begin{equation*}
\|f\|_1\le (4\pi^2-|E|)\frac{\|f\|_1}{8\pi^2}+|E|\|f\|_\infty.
\end{equation*}
After a simple transformation from this inequality we get (\ref{u4}).
\end{proof}
For any integer $n> 2$ of the form
\begin{equation*}
n=2^k+j,\quad 1\le j\le 2^k,\quad k=1,2,\ldots,
\end{equation*}
we denote
\begin{equation}\label{u5}
\bar n=2^{k-1}+\bigg[\frac{j+1}{2}\bigg],
\end{equation}
where $[\cdot ]$ stands for the integer part of a number.
A sequence of real valued functions $f_n(x,y)$, $n=2,3,\cdots , 2^m$, ($f_n\not\equiv 0$) is said to be a tree-system if
\begin{equation*}
\supp f_n\subset \{(x,y)\in \ZT^2:\, (-1)^{j+1}\cdot f_{\bar n}(x)>0\}.
\end{equation*}
The Haar system excluded the first function is the typical example of a tree-system. The following lemma was proved in \cite{Kar1}. Its Haar system case was considered in \cite{NiUl}.
\begin{lemma}\label{L1}
There exists a rearrangement $\sigma $ of the integers $\{2,3,\cdots ,2^m\}$ such that for any tree system $f_n(x,y)$, $n=2,3,\cdots ,2^m $, we have
\begin{equation*}
\sup_{2\le l\le 2^m}\left|\sum_{n=2}^{l}f_{\sigma (n)}(x,y)\right|\ge
\frac{1}{3}\sum_{n=2}^{2^m}\big|f_n(x,y)\big|.
\end{equation*}
\end{lemma}
For any $p$ integer we denote
\begin{align*}
&\delta_p^i=\left(\frac{2\pi (i-1)}{|p|},\frac{2\pi i}{|p|}\right),\\
&\delta_p^{i,j}=\delta_p^{i}\times \delta_p^{j}=\left(\frac{2\pi (i-1)}{|p|},\frac{2\pi i}{|p|}\right)\times \left(\frac{2\pi (j-1)}{|p|},\frac{2\pi j}{|p|}\right),\, i,j\in \ZZ,\\
&Q_p=\{\delta_p^{i,j}:\, 1\le i,j\le |p|\}.
\end{align*}
Given number $\gamma>1$ and sequence of integers $p_n$, $n=2,3,\ldots ,\nu=2^m$,  such that $1\le |p_2|<|p_3|<\ldots <|p_\nu|$ and $p_n|p_{n+1}$($p_n$ divides $p_{n+1}$) we associate the function system
\begin{equation}\label{u6}
a_n(x,y)=a_n(x,y)=\frac{\ZI_{E_n}(x,y)}{\sqrt{m}}e^{i(p_nx+q_n y)},\quad n\ge 2,
\end{equation}
where $E_n$ is defined to be the union of all rectangles $\delta\in Q_{p_n}$ satisfying the conditions
\begin{align}
&\overline{\delta}\subset \{(x,y)\in E_{\bar n} : (-1)^{j+1}
\cos (p_{\bar n}x+q_{\bar n}y)>0\},\label{u7}\\
 &\sup_{(x,y)\in\delta}\left| \sum_{k=1}^{\bar n} a_k(x,y)+\frac{e^{i(p_nx+q_n y)}}{\sqrt{m}}\right|\le \gamma.\label{u8}
\end{align}
Note that some of the sets $E_n$ can be empty. Besides, for the further convenience, we also assume that
\begin{equation*}
E_n=\varnothing, \quad n> \nu=2^m.
\end{equation*}
For any $k=1,2,\ldots,m-1$ we consider the collection
\begin{equation*}
\mathcal{E}_k=\{E_n:\,2^k<n\le 2^{k+1}\}.
\end{equation*}
The following relations give structural characterization of the sets $E_n$:
\begin{align}
&E_n\cap E_{n'}=\varnothing,\quad 2^k<n<n'\le 2^{k+1},\label{u9}\\
&E_n\in \mathcal{E}_k\Rightarrow E_{2n-1}, E_{2n}\in  \mathcal{E}_{k+1},\\
&E_{2n-1}\cup E_{2n}\subset E_n.\label{u11}
\end{align}
  Obviously the relations (\ref{u6})-(\ref{u8}) define the system (\ref{u6}) uniquely.
From (\ref{u9}) and (\ref{u11}) we obtain
\begin{equation}\label{u12}
\sum_{k=2}^\nu| a_n(x,y)|\le \sum_{k=2}^\nu\frac{\ZI_{E_n}(x,y)}{\sqrt m}\le \sqrt m.
\end{equation}
Thus we conclude that if $\gamma\ge \sqrt m$, then the condition (\ref{u8}) holds for any $\delta$, and so $ a_n(x,y)$ doesn't depend on $\gamma$. In this case the set $E_n$ and the function $a_n(x,y)$ will be denoted by $F_n$ and $b_n(x,y)$ respectively.
\begin{lemma}\label{L6}
If $p_n|p_{n+1}$ and $|p_{n+1}|\ge |p_n| \sqrt m $, $n=2,3,\ldots,\nu-1$, then
\begin{equation}\label{u13}
\left|\left\{\max_{2\le n\le \nu}\left|\sum_{k=2}^n b_k(x,y)\right|>\lambda \right\}\right|\le \frac{c}{\lambda},\quad \lambda >0,
\end{equation}
where $c>0$ is an absolute constant.
\end{lemma}
\begin{proof}
Without loss of generality we can assume that $p_n>0$, $n=2,3,\ldots,\nu$. Since $F_n$ consists of squares from $Q_{p_n}$, from (\ref{u6}) we get
\begin{equation*}
\int_{\delta_{p_n}^j} b_n(t,y)dt=0.
\end{equation*}
Define
\begin{equation}\label{u14}
f_n(x,y)=\frac{p_{n+1}}{2\pi}\sum_{j=1}^{p_{n+1}}\left[\int_{\delta_{p_{n+1}}^j} b_n(t,y)dt\cdot \ZI_{\delta_{p_{n+1}}^j}(x)\right]
\end{equation}
for $\, n=2,3,\ldots,\nu$, where $p_{\nu+1}>p_\nu \sqrt m$ is taken arbitrarily such that $p_\nu|p_{\nu+1}$. From the definition of $b_n(x,y)$ (see (\ref{u6})-(\ref{u8}) in the case $\gamma>\sqrt m$) we conclude
\begin{equation}\label{u15}
\supp f_n\subset  \supp  b_n\subset F_n,\quad n=2,3,\ldots,\nu.
\end{equation}
For a fixed $y$ the function $f_n(x,y)$ is constant on each interval $\delta_{p_{n+1}}^j$, $j=1,2,\ldots, p_{n+1}$, with respect to the variable $x$ and we have
\begin{align*}
\int_{\delta_{p_{n+1}}^j}f_{n+1}(t,y)dt&=\sum_{i:\,\delta_{p_{n+2}}^i\subset \delta_{p_{n+1}}^j}\int_{\delta_{p_{n+2}}^i} b_{n+1}(t,y)dt\\
&=\int_{\delta_{p_{n+1}}^i} b_{n+1}(t,y)dt=0.
\end{align*}
This means that the functions (\ref{u14}) form a martingale difference sequence with respect to $x$. Thus, applying a well known martingale inequality and then (\ref{u9}) and (\ref{u11}), we obtain
\begin{align}
\int_{\ZT^2}\max_{2\le n\le  \nu}\left|\sum_{k=2}^n f_k(x,y)\right|^2dxdy&\lesssim \sum_{k=2}^\nu \|f_k\|_2^2\le \sum_{k=2}^\nu \|b_k\|_2^2\label{u16}\\
&\le \frac{1}{m}\sum_{k=2}^\nu |E_k|\le 1.\nonumber
\end{align}
On the other hand if $x\in  \delta_{p_{k+1}}^j$, then
\begin{align}
|f_k(x,y)- b_k(x,y)|&\le\frac{p_{k+1}}{2\pi}\int_{\delta_{p_{k+1}}^j}| b_k(t,y)- b_k(x,y)|dt\\
&\le \sup_{|t-x|\le 2\pi/p_{k+1}}\frac{|\cos p_kt-\cos p_kx|}{\sqrt m}\le \frac{2\pi p_k}{p_{k+1}\sqrt m}\le \frac{2\pi}{m}.
\end{align}
Combining this with (\ref{u12}) and (\ref{u15}), we get
\begin{equation*}
\sum_{k=2}^\nu|f_k(x,y)- b_k(x,y)|\le \frac{2\pi}{m}\sum_{k=2}^\nu \ZI_{F_k}(x,y)\le 2\pi .
\end{equation*}
This together with (\ref{u16}) derives
\begin{align*}
\int_{\ZT^2}\max_{2\le n\le \nu}&\left|\sum_{k=2}^n b_k(x,y)\right|^2dxdy\\
&\le2\int_{\ZT^2}\max_{2\le n\le \nu}\left|\sum_{k=2}^nf_k(x,y)\right|^2dxdy+8\pi^2\lesssim1.
\end{align*}
Then, using Chebyshev's inequality, from this we will get (\ref{u13}).
\end{proof}
\begin{lemma}\label{L5}
Let $c$ be the constant from (\ref{u13}) and $\gamma\ge c+2$. Then if $2p_{\bar n}|p_n$ and
\begin{equation}\label{u17}
|p_{n}|>20\nu(|p_{\bar n}|+|q_{\bar n}|),
\end{equation}
then
\begin{equation}\label{u18}
\sum_{n=2}^{\nu } \int_{\ZT^2}   |\Re ( a_n(x,y))|dxdy>2\sqrt m.
\end{equation}
\end{lemma}
\begin{proof}
Observe that we may assume $p_n,q_n>0$, $n=2,3,\ldots,\nu$. The proof of (\ref{u18}) is based on the bound
\begin{equation}\label{u19}
\left|\bigcup_{n=2^{m-1}+1}^{2^m}E_n\right|>10.
\end{equation}
In order to prove (\ref{u19}) we define the sets  $U_n$ and $V_n$ to be the union of all squares $\delta\in Q_{p_n}$ satisfying the condition
\begin{align}
 &\bar\delta\subset \{(x,y)\in E_{\bar n}:\, (-1)^{j+1}\cos (p_{\bar n}x+q_{\bar n}y)>0\},\label{u20}\\
 &\bar\delta\cap\{(x,y)\in E_{\bar n}:\, (-1)^{j+1}\cos (p_{\bar n}x+q_{\bar n}y)>0\} \neq \varnothing,\label{u21}
\end{align}
respectively. We claim the following relations:
\begin{align}
&E_n\subset U_n \subset  E_{\bar n}\cap V_n,\label{u22}\\
&E_n\subset \overline{V_{2n}}\cup\overline{ V_{2n-1}},\label{u24}\\
&|V_n\setminus U_n|<\frac{1}{\nu},\label{u23}\\
&U_n\setminus E_n\subset \left\{(x,y)\in \ZT^2:\, \max_{1\le n\le \nu}\left|\sum_{k=1}^{n} a_k(x,y)\right|>\gamma-2\right\}.\label{u25}
\end{align}
The inclusion (\ref{u22}) immediately follows from the definitions of $E_n$, $U_n$ and $V_n$ (see (\ref{u7}), (\ref{u20}), (\ref{u21})).

From (\ref{u5}) it easily follows that $\overline{2n-1}=\overline{2n}=n$. Thus by the definition $V_{2n}$ and $V_{2n-1}$ are the union of squares $\delta\in Q_{p_n}$ satisfying respectively
\begin{align*}
&\bar\delta\cap\{(x,y)\in E_n:\, \cos (p_{\bar n}x+q_{\bar n}y)>0\} \neq \varnothing,\\
&\bar\delta\cap\{(x,y)\in E_n:\, \cos (p_{\bar n}x+q_{\bar n}y)<0\} \neq \varnothing.
\end{align*}
This implies (\ref{u24}).

To prove (\ref{u23}) we note that the set $V_n\setminus U_n$ ($n=2^k+j$) consists of the squares $\delta\in Q_{p_n}$ satisfying
\begin{equation}\label{u26}
\bar\delta\cap \{(x,y)\in E_{\bar n}:\, (-1)^{j+1}\cos (p_{\bar n}x+q_{\bar n}y)=0\}\neq \varnothing.
\end{equation}
Then observe that if
\begin{equation}\label{u27}
\delta_{p_n}^{\alpha_1,\beta}\subset  V_n\setminus U_n
\end{equation}
and the integer $\alpha_2$ satisfies
\begin{equation}\label{u28}
\alpha_1+\frac{p_n}{2 p_{\bar n}}\cdot \frac{1}{3\nu}+1<\alpha_2<\alpha_1+\frac{p_n}{2p_{\bar n}}\left(1-\frac{1}{3\nu}\right)-1,
\end{equation}
then
\begin{equation}\label{u29}
\delta_{p_n}^{\alpha_2,\beta}\cap (V_n\setminus U_n) = \varnothing .
\end{equation}
Indeed, suppose we have (\ref{u27}), (\ref{u28}), and besides (\ref{u29}) doesn't hold. Then according to (\ref{u26}) there are points
\begin{equation}\label{u30}
(x_1,y_1)\in \overline{\delta_{p_n}^{\alpha_1,\beta}},\quad (x_2,y_2)\in \overline{\delta_{p_n}^{\alpha_2,\beta}},
\end{equation}
such that $\cos (p_{\bar n}x_1+q_{\bar n}y_1)=\cos (p_{\bar n}x_2+q_{\bar n}y_2)=0$ and therefore
\begin{equation*}
p_{\bar n}x_1+q_{\bar n}y_1=\frac{\pi}{2}+\pi l_1,\quad p_{\bar n}x_2+q_{\bar n}y_2=\frac{\pi}{2}+\pi l_2,
\end{equation*}
for some integers $l_1$ and $l_2$. Thus we will get
\begin{equation}\label{u31}
p_{\bar n}(x_2-x_1)+q_{\bar n}(y_2-y_1)=\pi (l_2-l_1).
\end{equation}
From (\ref{u17}), (\ref{u28}) and (\ref{u30}) we derive
\begin{multline}\label{u32}
|p_{\bar n}(x_2-x_1)+q_{\bar n}(y_2-y_1)|\le\frac{2\pi p_{\bar n}(\alpha_2-\alpha_1+1)}{p_n}+\frac{2\pi q_{\bar n}}{p_n}\\
< \pi\left(1-\frac{1}{3\nu}\right)+\frac{\pi}{10\nu}<\pi.
\end{multline}
Combining (\ref{u31}) and (\ref{u32}), we get $l_1=l_2$ and therefore
\begin{equation}\label{u33}
p_{\bar n}(x_2-x_1)=q_{\bar n}(y_1-y_2).
\end{equation}
On the other hand, using (\ref{u17}), (\ref{u28}) and (\ref{u30}), we have
\begin{equation}\label{u34}
q_{\bar n}(y_1-y_2)\le \frac{2\pi q_{\bar n}}{p_n}<\frac{\pi}{10\nu}
\end{equation}
and
\begin{equation}\label{u35}
p_{\bar n}(x_2-x_1)>\frac{2\pi p_{\bar n}(\alpha_2-\alpha_1-1)}{p_n}>\frac{2\pi p_{\bar n}}{p_n}\frac{p_n}{2 p_{\bar n}}\cdot \frac{1}{3\nu}=\frac{\pi}{3\nu}.
\end{equation}
Combining (\ref{u33})-(\ref{u35}) we get contradiction. Hence we have
\begin{equation}\label{u36}
(\ref{u27},(\ref{u28} \Rightarrow (\ref{u29}.
\end{equation}
For $1\le i\le p_n$ and $1\le k\le p_{\bar n}$ we consider the squares
\begin{equation}\label{u37}
\left\{\delta_{p_n}^{i,j}:\,\frac{(k-1)p_n}{2p_{\bar n}}<j\le\frac{kp_n}{2p_{\bar n}}\right\}.
\end{equation}
Since $2p_{\bar n}$ divides $p_n$, either they are all inside of $E_{\bar n}$ or all are outside of $E_{\bar n}$.
Using the relations (\ref{u17}) and (\ref{u36}),  one can easily conclude that the number of squares $\delta_{p_n}^{i,j}$ from the collection (\ref{u37}) included in $V_n\setminus U_n$ doesn't exceed the quantity
\begin{equation*}
\frac{2}{3\nu}\cdot \frac{p_n}{2p_{\bar n}}+3<\frac{1}{\nu}\cdot \frac{p_n}{2p_{\bar n}}.
\end{equation*}
Thus we get the number of all squares $\delta \in Q_{p_n}$ with $\delta\subset V_n\setminus U_n$ is estimated above by the value $(p_n)^2/\nu$, where $(p_n)^2$ is the number of all squares $\delta \in Q_{p_n}$. Hence we get
\begin{equation*}
|V_n\setminus U_n|<\frac{(p_n)^2}{\nu}\cdot \frac{4\pi^2}{(p_n)^2}=\frac{1}{\nu}.
\end{equation*}
If $(x,y)$ belong to the left side of (\ref{u25}), then according to the definitions of $U_n$ ((\ref{u20})) and $E_n$ ((\ref{u7}), (\ref{u8}))), there exists a unique $\delta\in Q_{p_{n}}$  such that $(x,y)\in \delta\subset U_n$ and $\delta\cap E_n=\varnothing$. From (\ref{u20}) we have $\delta\subset E_{\bar n}$, then using (\ref{u7} and (\ref{u8}, we conclude
\begin{align}
&\sup_{(u,v)\in \delta}\left| \sum_{k=1}^{\bar n} a_k(u,v)\right|\le  \gamma,\label{u38}\\
&\sup_{(u,v)\in \delta}\left|\sum_{k=1}^{\bar n} a_k(u,v)+\frac{e^{i(p_nu+q_n v)}}{\sqrt{m}}\right|>\gamma.\label{u39}
\end{align}
By (\ref{u39}) there exists a point $(x_0,y_0)\in \delta$ satisfying
\begin{equation}\label{u40}
\left|\sum_{k=1}^{\bar n} a_k(x_0,y_0)+\frac{e^{i(p_nx_0+q_n y_0)}}{\sqrt{m}}\right|>\gamma.
\end{equation}
On the other hand for an arbitrary $(x,y),(x',y')\in \delta$ we have
\begin{align*}
\left|e^{p_{\bar  n}x+q_{\bar  n}y}-e^{p_{\bar  n}x'+q_{\bar  n}y'}\right|&\\
\le &\left|e^{p_{\bar  n}x+q_{\bar  n}y}-e^{p_{\bar  n}x+q_{\bar  n}y'}\right|+\left|e^{p_{\bar  n}x+q_{\bar  n}y'}-e^{p_{\bar  n}x'+q_{\bar  n}y'}\right|\\
= &\left|e^{q_{\bar  n}y}-e^{q_{\bar  n}y'}\right|+\left|e^{p_{\bar  n}x}-e^{p_{\bar  n}x'}\right|\\
\le &\sqrt 2(q_{\bar  n}|y-y'|+p_{\bar  n}|x-x'|)\\
\le &\frac{4\pi(p_{\bar  n}+q_{\bar  n})}{p_{ n}}<\frac{1}{\nu}
\end{align*}
and therefore
\begin{equation*}
r=\sup_{(x,y),(x',y')\in \delta}\left| \sum_{k=1}^{\bar n} a_k(x,y)-\sum_{k=1}^{\bar n} a_k(x',y')\right|\le 1.
\end{equation*}
Thus, using (\ref{u40}), we obtain
\begin{align*}
\left|\sum_{k=1}^{\bar n} a_k(x,y)\right|&\ge \left| \sum_{k=1}^{\bar n} a_k(x_0,y_0)\right|-r\\
&\ge \left| \sum_{k=1}^{\bar n} a_k(x_0,y_0)+\frac{e^{p_nx_0+q_ny_0}}{\sqrt m}\right|-\frac{|e^{p_nx_0+q_ny_0|}}{\sqrt m}-1\\
&> \gamma-2,
\end{align*}
which gives (\ref{u25}). According to (\ref{u24}) we have
\begin{equation*}
E_n\setminus (E_{2n-1}\cup E_{2n})\subset (\overline{V_{2n-1}}\setminus E_{2n-1})\cup (\overline{V_{2n}}\setminus E_{2n})
\end{equation*}
Using this we get
\begin{align}
\sum_{n=2^{m-1}+1}^{2^m}E_n&=\ZT^2\setminus \bigcup_{n=2}^{2^{m-1}}(E_n\setminus (E_{2n-1}\cup E_{2n}))\label{u41}\\
&\supset \ZT^2\setminus \bigcup_{n=3}^{\nu}   (\overline{V_{n}}\setminus E_{n})\nonumber\\
&= \ZT^2\setminus \left(\bigcup_{n=3}^{\nu}   (\overline{V_{n}}\setminus U_{n})\cup (U_{n}\setminus E_n)\right).\nonumber
\end{align}
From (\ref{u23}) we get
\begin{equation*}
\left|\bigcup_{n=3}^{\nu}\overline{V_n}\setminus U_n\right|=\left|\bigcup_{n=3}^{\nu}V_n\setminus U_n\right|\le  1.
\end{equation*}
Applying Lemma \ref{L6}, from (\ref{u25}) it follows that
\begin{align}
&\left|\bigcup_{n=3}^{\nu}(U_n\setminus E_n) \right|\label{u42}\\
&\qquad\le \left| \left\{(x,y)\in \ZT^2:\, \max_{1\le n\le \nu}\left|\sum_{k=1}^{n} a_k(x,y)\right|>\gamma-2\right\}\right|\nonumber\\
&\qquad\le \frac{c}{\gamma-2}\le \frac{c+2}{\gamma}<1.\nonumber
\end{align}
Combing (\ref{u41})-(\ref{u42}) we obtain
\begin{equation*}
\left|\sum_{n=2^{m-1}+1}^{2^m}E_n\right|\ge |\ZT^2|-2\ge 10
\end{equation*}
and so (\ref{u19}). Using the properties of the sets $E_n$ ((\ref{u9})-(\ref{u11})), from (\ref{u19}) we conclude
\begin{equation*}
\sum_{n=2^{k-1}+1}^{2^k}|E_n|\ge \sum_{n=2^{m-1}+1}^{2^m}|E_n|>10
\end{equation*}
for any $k=1,2,\ldots ,m$ and therefore
\begin{equation}\label{u43}
\sum_{n=2}^\nu |E_n|\ge 10m.
\end{equation}
On the other hand, since $E_n$ is a union of squares $\delta \in Q_{p_n}$, we have
\begin{align}\label{u44}
 \int_{E_n}|\cos(p_nx+q_ny)|&dxdy\\
 &=\sum\limits_{\delta \in Q_{p_n},\, \delta\subset E_n}\int_{\delta}|\cos(p_nx+q_ny)|dxdy\nonumber\\
 &=\sum\limits_{\delta \in Q_{p_n},\, \delta\subset E_n}\frac{1}{p_nq_n}\int_{\ZT^2}|\cos(x+y)|dxdy\nonumber\\
 &=4\pi\sum\limits_{\delta \in Q_{p_n},\, \delta\subset E_n}\frac{1}{p_nq_n}\nonumber\\
 &=\frac{1}{\pi} \sum\limits_{\delta \in Q_{p_n},\, \delta\subset E_n}|\delta|=\frac{|E_n|}{\pi} .\nonumber
\end{align}
Combining (\ref{u43}) and (\ref{u44}) we obtain
\begin{align*}
\sum_{n=2}^{\nu } \int_{\ZT^2}  |\Re(a_n(x,y))|dxdy&=\frac{1}{\sqrt m}\sum_{n=2}^{\nu }  \int_{E_n}|\cos(p_nx+q_ny)|dxdy\\
&\ge \frac{1}{\pi \sqrt m}\sum_{n=2}^{\nu } |E_n|\ge 2\sqrt m.
\end{align*}
\end{proof}
\begin{lemma}\label{L7}
Let $\sigma$ be the rearrangement of the integers $\{2,3,\cdots ,2^m\}$ determined by Lemma \ref{L1}. If $\gamma $ and $p_n$ satisfy the hypothesis of Lemma \ref{L5}, then
\begin{equation}\label{u45}
\left|\left\{(x,y)\in {\ZT^2}:\max_{2\le n\le \nu }\left|\sum_{j=1}^n a_{\sigma (j)}(x,y)\right|>\frac{\sqrt{m }}{120}\right\}\right|>1.
\end{equation}
\end{lemma}
\begin{proof}
The function system
\begin{equation*}
u_n(x,y)=\frac{\ZI_{E_n}(x,y)}{\sqrt m}\cos(p_nx+q_ny)=\Re (a_n(x,y)),
\end{equation*}
where $n=2,3,\cdots ,\nu =2^m$, is a tree-system, since by definition ((\ref{u7}), (\ref{u8})) we have
\begin{align*}
\supp  (u_n(x,y))\subset E_n &\subset \{(x,y)\in E_{\bar n} :\, (-1)^{j+1}
    \cos (p_{\bar n}x+q_{\bar n}y)>0\big\}\\
    &=\{(x,y)\in E_{\bar n} :\,
    (-1)^{j+1}u_{\bar n}(x,y)>0\big\} \\
    &=\{(x,y)\in {\ZT^2}:\,(-1)^{j+1}u_{\bar n}(x,y)>0\big\},
\end{align*}
Hence, applying Lemma \ref{L1}, we get
\begin{multline}\label{u47}
\sup_{2\le l\le \nu}\left|\sum_{n=2}^{l} a_{\sigma (n)}(x,y)\right|\\
\ge\sup_{2\le l\le \nu}\left|\sum_{n=2}^{l}u_{\sigma (n)}(x,y)\right|
\ge\frac{1}{3}\sum_{n=2}^{\nu}\big|u_{n}(x,y)\big|.
\end{multline}
Consider  the function
\begin{equation}\label{u48}
 f(x,y)= \sum_{n=2}^{\nu}|u_{n}(x,y)|.
\end{equation}
By Lemma \ref{L6} we have $\|f\|_1>2\sqrt m$. On the other hand $\|f\|_\infty \le \sqrt m$, because
\begin{equation*}
0\le  f(x,y)\le \sum_{n=2}^{\nu}\frac{\ZI_{E_n}(x,y)}{\sqrt m}\le \sqrt m.
\end{equation*}
Thus, applying Lemma \ref{L}, we get
\begin{align*}
\left|\left\{(x,y)\in {\ZT^2}:\,|f(x,y)|>\frac{\sqrt m}{4\pi^2}\right\}\right|&>\frac{4\pi^2}{8\pi^2(\|f\|_\infty/\|f\|_1)-1}\\
&\ge \frac{4\pi^2}{4\pi^2-1}>1.
\end{align*}
Combining this with (\ref{u47}) and (\ref{u48}), we obtain (\ref{u45}).
\end{proof}
\begin{lemma}\label{L8}
If $\gamma>1$ and $p_n$ divides $p_{n+1}$, then
\begin{equation*}
\left|\sum_{k=2}^\nu a_k(x,y)\right|\le \gamma,\quad (x,y)\in \ZT^2.
\end{equation*}
\end{lemma}
\begin{proof}
Take an arbitrary $(x,y)\in \ZT^2$. We have
\begin{equation}\label{y61}
(x,y)\in E_n\setminus (E_{2n-1}\cup E_{2n})
\end{equation}
for an integer $2\le n\le 2^{m}$. Then from the definition of the sets $E_n$ ((\ref{u7}), (\ref{u8})) we get
\begin{equation*}
\left|\sum_{k=2}^{ n} a_k(x,y)\right|\le \gamma,
\end{equation*}
and $ a_k(x,y)=0$ while $k>n$. This implies
\begin{equation*}
\left|\sum_{k=2}^\nu a_k(x,y)\right|=\left|\sum_{k=2}^{ n} a_k(x,y)\right|\le \gamma,
\end{equation*}
and so the lemma is proved.
\end{proof}

Any finite sum of the form
\begin{equation*}
T(x,y)=\sum_{(n,m)\in G}c_{nm}e^{i(nx+my)},
\end{equation*}
where $G\subset \ZR^2$ is a bounded region, is said to be a double trigonometric polynomial. The spectrum of this polynomial is denoted by
\begin{equation*}
\spec (T)=\left\{(n,m)\in \ZZ^2:\, c_{nm}\neq0\right\}.
\end{equation*}
We will consider the sectorial regions
\begin{align}\label{u49}
V(\alpha,\beta)=\{(x,y)\in \ZR^2:\, x=r\cos \theta , &y=r\sin\theta,\\
&r\ge 0,\, \alpha\le
\theta < \beta \},\nonumber
\end{align}
where $0\leq \alpha<\beta\le 2\pi$. In the proof of the following basic lemma the technique of the paper \cite {Kar1} is used.
\begin{lemma}\label{L4}
If $S_n$, $n=2,3,\cdots ,\nu =2^m$, $m\ge 10$, is an arbitrary sequence of sectors of the form (\ref{u49}), then there exists a sequence of polynomials $T_n(x,y)$, $n=2,3,\cdots ,\nu$, such that
\begin{align}
& \spec T_n\subset S_n,\quad n=2,3,\cdots ,\nu,\label{u50}\\
&\left\|\sum_{n=1}^\nu T_n\right\|_\infty\le c_1,\label{u51}\\
&\left|\left\{(x,y)\in {\ZT^2}:\max_{2\le n\le \nu }\left|\sum_{j=1}^nT_j(x,y)\right|>c_2\sqrt m\right\}\right|>1,\label{u52}
\end{align}
where $c_1,c_2>0$ are some absolute constants.
\end{lemma}
\begin{proof} Let $\sigma$ be the rearrangement of the numbers $\{2,3,\cdots ,\nu =2^m\}$ determined by Lemma \ref{L1} and let
\begin{equation*}
\gamma=c+2,
\end{equation*}
where $c$ is the constant from (\ref{u13}). We define positive integers $p_n,q_n$  and double trigonometric polynomials $f_n(x,y)$ such that they, together with the  sets $E_n\subset \ZT^2$, $n=2,3,\cdots \nu$, defined by (\ref{u7}) and (\ref{u8}) satisfy the relations
\begin{align}
&2p_{\bar n}| p_n,\label{u53}\\
& |p_{n}|>20m(|p_{\bar n}|+|q_{\bar n}|),\label{u54}\\
&\spec \left(f_n(x,y)e^{i(p_nx+q_ny)}\right) \subset S_{\sigma^{-1}(n)},\, 2\le n\le\nu ,\label{u55}\\
&0\le f_n(x,y)\le \frac{1}{\sqrt {m}}, \quad (x,y)\in \ZT^2,\label{u56}\\
&\left|f_n(x,y)-\frac{\ZI_{E_n}(x,y)}{\sqrt {m}}\right|< \frac{1}{\nu} ,\quad (x,y)\in E_n\cup (E_{\bar{n}})^c, \, 2\le n\le\nu.\label{u57}\
\end{align}
We will use the induction. As a first step of induction we suppose $E_2={\ZT^2}$, $f_2(x,y)\equiv 1/\sqrt m$ and fix integers $p_2$ and $q_2$ with $(p_2,q_2)\in S_{\sigma^{-1}(2)} $. Obviously we will have the relations (\ref{u53})-(\ref{u57}) for $n=2$.
Now we suppose that the conditions (\ref{u53})-(\ref{u57}) are satisfied for any $n<l$ and in particular for $\bar l$. Since $\overline  E_l \subset E_{\bar l}$ and $E_{\bar l}$ is an open set, we may find a polynomial $f_l(x,y)$ satisfying (\ref{u56}) and (\ref{u57}) (with $n=l$). Since
\begin{equation*}
\spec \left(f_n(x,y)e^{i(p_nx+q_ny)}\right)=\spec \left(f_n\right)+(p_n,q_n),
\end{equation*}
we can choose integers $p_n$ and $q_n$ satisfying (\ref{u53})-(\ref{u55}) ($n=l$). This completes the induction.
Now we define our desired polynomials as follows:
\begin{equation*}
T_n(x,y)=f_{\sigma(n)}(x,y)e^{i(p_{\sigma(n)}x+q_{\sigma(n)}y)},\,2\le n\le \nu.
\end{equation*}
Together with $T_n(x,y)$ we will consider also the function system $a_n(x,y)$ defined in (\ref{u6}). From (\ref{u55}) we have
\begin{equation*}
\spec(T_n)=\spec \left(f_{\sigma(n)}(x,y)e^{i(p_{\sigma(n)}x+q_{\sigma(n)}y)}\right)\subset S_n,
\end{equation*}
which implies (\ref{u50}).
If $(x,y)\in \ZT^2$, then we have
\begin{equation*}
(x,y)\in E_n\setminus (E_{2n-1}\cup E_{2n})
\end{equation*}
for some integer $2\le n\le 2^{m}$.
From this it follows that $(x,y)\in E_j\cup (E_{\bar{j}})^c$ whenever $2\le j\le \nu$ and $j\neq 2n,2n-1$. Thus we get
\begin{align*}
\left|f_j(x,y)-\frac{\ZI_{E_j}(x,y)}{\sqrt m}\right|< \frac{1}{\nu},\quad 2\le j\le \nu,\quad j\neq 2n,2n-1,
\end{align*}
and therefore
\begin{align}\label{u58}
\sum_{j=1}^\nu &\left|T_j(x,y)-a_{\sigma(n)}(x,y)\right|\\
&\qquad\qquad\qquad=\sum_{j=1}^\nu \left|f_j(x,y)-\frac{\ZI_{E_j}(x,y)}{\sqrt m}\right|\nonumber\\
&\qquad\qquad\qquad\le \sum_{j\neq 2n,2n-1} \left|f_j(x,y)-\frac{\ZI_{E_j}(x,y)}{\sqrt m}\right|+1\nonumber\\
&\qquad\qquad\qquad\le 2.\nonumber
\end{align}
From this and Lemma \ref{L8} we get
\begin{align*}
\left|\sum_{n=1}^\nu T_n(x,y)\right| \le \left|\sum_{n=1}^\nu a_n(x,y)\right|+2\le \gamma+2=c+4
\end{align*}
and hence (\ref{u51}). From (\ref{u58}) it also follows that
\begin{align}\label{u59}
\max_{2\le l\le \nu}\left|\sum_{j=2}^lT_j(x,y)\right| &\ge\max_{2\le l\le\nu}\left|\sum_{j=2}^la_{\sigma (j)}(x,y)\right|\\
&\qquad\qquad-\sum_{j=2}^\nu | T_j(x,y)-a_{\sigma (j)}(x,y)|\nonumber\\
&\ge \max_{2\le l\le\nu}\left|\sum_{j=2}^la_{\sigma (j)}(x,y)\right|-2.\nonumber
\end{align}
Combining (\ref{u59}) with Lemma \ref{L7}, we obtain
\begin{align*}
&\left|\left\{(x,y)\in {\ZT^2}:\max_{2\le l\le \nu }\left|\sum_{j=2}^lT_n(x,y)\right|>\frac{\sqrt{m }}{240}\right\}\right|\\
&>\left|\left\{(x,y)\in {\ZT^2}:\max_{2\le l\le \nu }\left|\sum_{j=2}^la_{\sigma (n)}(x,y)\right|>\frac{\sqrt{m }}{120}\right\}\right|>1.
\end{align*}
and therefore we will have (\ref{u52}).
\end{proof}
\begin{lemma}\label{L11}
 For any $\delta>0$ and $m, s\in \ZN$ there exist a sequence of regions $\Delta_n$ of the form (\ref{v4}) and polynomials $Q_n$, $n=2,3,\cdots ,\nu=2^m$, such that
\begin{align}
&\rho(\Delta_n)<1+\delta,\quad n=2,3,\cdots ,\nu,\label{u60}\\
&\Delta_1=\Delta(s,s), \quad  \Delta_{n}\subset \Delta_{n+1},\quad n=1,2,,\cdots ,\nu-1,\label{u61}\\
& \spec Q_n\subset (\Delta_n\setminus\Delta_{n-1})\cap \ZR_+^2,\quad n=2,3,,\cdots ,\nu,\label{u62}\\
&\left\|\sum_{j=2}^\nu Q_j\right\|_\infty\le c_1,\label{u63}\\
&\left|\left\{(x,y)\in {\ZT^2}:\max_{2\le n\le \nu }\left|\sum_{j=2}^nQ_j(x,y)\right|>c_2\sqrt{m }\right\}\right|>1.\label{u64}
\end{align}
\end{lemma}
\begin{proof}
Consider the sectors
\begin{equation*}
V_n=V\left(\frac{3\pi}{4}-\frac{\varepsilon}{n},\pi\right)\subset (-\infty,0]\times [0,+\infty),\quad n=2,3,\cdots ,\nu,
\end{equation*}
of the form (\ref{u49}) and set $S_n=V_n\setminus V_{n-1}$. Applying Lemma \ref{L4}, we find polynomials $T_n(x,y)$ with the properties (\ref{u50})-(\ref{u52}). We have $\spec (T_n)\subset [-l,0]\times [0,l]$, $n=2,3,\cdots ,\nu$, for some integer $l>s$. Denote
\begin{align*}
&Q_n(x,y)= T_n(x,y)e^{ilx},\\
&\Theta_n=((l,0)+V_n)\cap \ZR_+^2,\quad n=2,3,\cdots ,\nu.
\end{align*}
Observe that $\Theta_n$ are triangles of the form (\ref{v5}), and the regions
\begin{align*}
\Delta_n=\Theta_n&\cup\{(x,y)\in \ZR^2:\, (-x,y)\in \Theta_n\}\\
&\cup
\{(x,y)\in \ZR^2:\, (-x,-y)\in \Theta_n\}\\
&\cup\{(x,y)\in \ZR^2:\, (x,-y)\in \Theta_n\}
\end{align*}
have the form (\ref{v4}). It is clear that small enough number $\varepsilon$ guarantees (\ref{u60}). The conditions (\ref{u61}) and  (\ref{u62}) are immediate. The relation (\ref{u63}) follows from (\ref{u51}), (\ref{u64}) follows from (\ref{u52}), since we have
\begin{equation*}
\left|\sum_{j=2}^n Q_j(x,y)\right|=\left|\sum_{j=2}^n T_j(x,y)\right|.
\end{equation*}
Lemma is proved.
\end{proof}
The following lemma is a version of Lemma \ref{L11} for balls instead of triangles. It will be used in the proof of Theorem \ref{T2}.
\begin{lemma}\label{L12}
 For any $\delta>0$ and $m, r\in \ZN$ there exist a sequence of balls $U_n$ and polynomials $Q_n$, $n=2,3,\cdots ,\nu=2^m$, such that
\begin{align}
&\tau(U_n)<\delta,\quad n=2,3,\cdots ,\nu,\label{u65}\\
&B(0,0,r)\supset U_n,\quad n=2,3,,\cdots ,\nu,\label{u66}\\
& \spec Q_n\subset \left(U_n\setminus\bigcup_{k=2}^{n-1}U_{k}\right),\quad n=2,3,,\cdots ,\nu,\label{u67}\\
&\left\|\sum_{j=2}^\nu Q_j\right\|_\infty\le c_1,\label{u68}\\
&\left|\left\{(x,y)\in {\ZT^2}:\max_{2\le n\le \nu }\left|\sum_{j=2}^nQ_j(x,y)\right|>c_2\sqrt{m }\right\}\right|>1.\label{u69}
\end{align}
\end{lemma}
\begin{proof}
Consider the sectors
\begin{equation*}
V_n=V\left(\theta_{2n+1},\theta_{2n}\right),\quad n=2,3,\cdots ,\nu,
\end{equation*}
where
\begin{equation*}
\theta_k=\frac{\pi}{2}+\frac{\varepsilon}{k}.
\end{equation*}
Applying Lemma \ref{L4}, we find polynomials $T_n(x,y)$ with the properties (\ref{u50})-(\ref{u52}).
Denote
\begin{align*}
&Q_n(x,y)= T_n(x,y)e^{iRx},\\
&\Theta_n=(R,0)+V_n,\quad n=2,3,\cdots ,\nu,
\end{align*}
where the number $R$ will be determined bellow. We have
\begin{equation*}
\spec (Q_n)\subset \Theta_n.
\end{equation*}
Consider the balls
\begin{equation*}
B_n=B\left(0,-R\tan(\varepsilon/n),R/\cos (\theta_{n})\right)
\end{equation*}
and the lines $L_n$ given by the formulae
\begin{equation*}
x=R+t\cos(\theta_{n}), \quad y=t\sin(\theta_{n}).
\end{equation*}
Note that the boundary of the sector $\Theta_n$ is determined by the lines $L_{2n}$ and $L_{2n-1}$. Besides, the line $L_n$ is tangential for the ball $B_n$ at the point $(R,0)$. A simple calculation shows that
\begin{equation}\label{u70}
\tau(B_n)=\sin (\varepsilon/n)<\sin \varepsilon.
\end{equation}
Using this, for a bigger enough $R$ we will have a good approximation of the lines $L_n$ by the balls $B_n$ and therefore we will get
\begin{align}
&\Theta_n\subset B_{k},\quad  k>n,\label{u71}\\
&\Theta_n\cap B_{k}=\varnothing,\quad  k\le n.\label{u72}
\end{align}
We denote $U_n=B_{2n}$. A small enough number $\varepsilon$ guarantees (\ref{u65}) according to (\ref{u70}). Then taking $R$ bigger enough we derive (\ref{u66}). It easy to check that the relation (\ref{u67}) can be obtained from (\ref{u71}) and (\ref{u72}).  Then (\ref{u68})  follows from (\ref{u51}), (\ref{u69}) follows from (\ref{u52}), since we have
\begin{equation*}
\left|\sum_{j=2}^n Q_j(x,y)\right|=\left|\sum_{j=2}^n T_j(x,y)\right|.
\end{equation*}
Lemma is proved.
\end{proof}

\section{Proof of theorems}

\begin{proof}[Proof of Theorem \ref{T1}]
Fix integers $\mu_k$, $k=1,2,\ldots $ satisfying
\begin{equation}\label{p1}
\mu_{k+1}-\mu_{k}=2^{k^6}-1,\quad k=1,2,\ldots.
\end{equation}
Applying Lemma \ref{L11} successively, we may define polynomials $Q_n(x,y)$ and regions $\Delta_n$, $n=1,2,\ldots$, of the form (\ref{v4}) satisfying the relations
 \begin{align}
&\rho(\Delta_n)<1+1/n,\quad n=1,2,\cdots ,\label{p2}\\
&\Delta_{n-1}\subset \Delta_n,\quad n=1,2,\cdots ,\label{p3}\\
& \spec Q_n\subset \left(\Delta_n\setminus\Delta_{n-1}\right)\cap \ZR_+^2,\quad n=1,2,\cdots ,\label{p4}\\
&\left\|\sum_{j=\mu_k+1}^{\mu_{k+1}} Q_j\right\|_\infty <c_1,\quad k=1,2,\ldots ,\label{p5}\\
&\left|\left\{(x,y)\in {\ZT^2}:\max_{\mu_k< n\le\mu_{k+1} }\left|\sum_{j=\mu_k+1}^nQ_j(x,y)\right|>c_2k^3\right\}\right|>1.\label{p6}
\end{align}
It is clear that we may define $g_n(x,y)$ equal to one of the following real polynomials
\begin{equation*}
\Re(Q_n(x,y)), \quad \Im(Q_n(x,y)),
\end{equation*}
such that
\begin{equation}\label{p7}
\left|\left\{(x,y)\in {\ZT^2}:\max_{\mu_k< n\le\mu_{k+1} }\left|\sum_{j=\mu_k+1}^ng_j(x,y)\right|>\frac{c_2k^3}{2}\right\}\right|>\frac{1}{2}.
\end{equation}
Each $g_n(x,y)$ can be considered as a complex polynomial and from (\ref{p4}) it follows that
\begin{equation*}
\spec g_n\subset \Delta_n\setminus\Delta_{n-1},\quad n=1,2,\ldots .
\end{equation*}
Consider a real function
\begin{equation}\label{p8}
f(x,y)=\sum_{k=1}^\infty\frac{1}{k^2}\sum_{j=\mu_k+1}^{\mu_{k+1}}g_j(n_kx,n_ky),
\end{equation}
where $n_k\nearrow\infty $ is a sequence of integers that will be defined bellow. Since this series converges uniformly, $f$ is continuous. We denote
\begin{align*}
&E_k=\left\{(x,y)\in {\ZT^2}:\max_{\mu_k< n\le\mu_{k+1} }\left|\sum_{j=\mu_k+1}^ng_j(x,y)\right|>\frac{c_2k^3}{2}\right\},\\
&\tilde\Delta_j=\{(u,v)\in\ZR^2:\, (u/n_k,v/n_k)\in\Delta_j\},\quad \mu_k<j\le \mu_{k+1}.
\end{align*}
It is clear that each $\tilde\Delta_j$ is also a region of the form (\ref{v4}) and from (\ref{p2})-(\ref{p4}), (\ref{p7}) we get respectively
\begin{align}
&\rho(\tilde\Delta_n)<1/n,\quad n=1,2,\cdots ,\label{p9}\\
&\tilde\Delta_{n-1}\subset \tilde\Delta_n,\quad n=1,2,\cdots ,\label{p10}\\
& \spec \big(g_n(n_kx,n_ky)\big)\subset \tilde\Delta_n\setminus\tilde\Delta_{n-1},\, \mu_k<n\le \mu_{k+1},\label{p11}\\
&|E_k|>1/2.\label{p14}
\end{align}
According to Lemma \ref{L22}, we may define the integers $n_k$ such that
\begin{equation}\label{p12}
\left|\cap_{l\ge 1}\cup_{k\ge l}E_k(n_k)\right|=4\pi^2.
\end{equation}
It is clear that if  $(x,y)\in E_k(n_k)$, then
\begin{equation}\label{p13}
\max_{\mu_k< n\le\mu_{k+1} }\left|\sum_{j=\mu_k+1}^ng_j(n_kx,n_ky)\right|> \frac{c_2 k^3}{2}.
\end{equation}
From (\ref{p10}) and (\ref{p11}) it follows that
\begin{equation*}
\left|S_{\tilde\Delta_n}(x,y,f)-S_{\tilde\Delta_{\mu_k}}(x,y,f)\right|=\frac{1}{k^2}\left|\sum_{j=\mu_k+1}^ng_j(n_kx,n_ky)\right|
\end{equation*}
for any $\mu_k< n\le\mu_{k+1} $. Combining this with (\ref{p13}), we get
\begin{equation*}
\max_{\mu_k< n\le\mu_{k+1}}\left|S_{\tilde\Delta_n}(x,y,f)-S_{\tilde\Delta_{\mu_k}}(x,y,f)\right|> \frac{c_2k}{2},\quad (x,y)\in E_k(n_k),
\end{equation*}
then taking into account of (\ref{p12}), we find that the sums $S_{\tilde\Delta_n}(x,y,f)$ diverge almost everywhere. This proves Theorem \ref{T1}.
\end{proof}
In the proof of Theorem \ref{T2} we use the same argument as in Theorem \ref{T1}. So the details of the proof will be omitted.
\begin{proof}[Proof of Theorem \ref{T2}]
Fix integers $\mu_k$, $k=1,2,\ldots $, satisfying (\ref{p1}).
Applying Lemma \ref{L11}, we define polynomials $Q_n(x,y)$ and balls $B_n$, $n=1,2,\ldots$, satisfying the relations
\begin{align*}
&\tau(U_n)<1/n,\quad n=2,3,\cdots ,\nu,\\
& \spec Q_n\subset \left(U_n\setminus\bigcup_{k=2}^{n-1}U_{k}\right),\quad n=1,2,\cdots ,\\
&\left\|\sum_{j=\mu_k+1}^{\mu_{k+1}} Q_j\right\|_\infty <c_1,\quad k=1,2,\ldots ,\\
&\left|\left\{(x,y)\in {\ZT^2}:\max_{\mu_k< n\le\mu_{k+1} }\left|\sum_{j=\mu_k+1}^nQ_j(x,y)\right|>c_2k^3\right\}\right|>1.
\end{align*}
Then the function
\begin{equation*}
f(x,y)=\sum_{k=1}^\infty\frac{1}{k^2}\sum_{j=\mu_k+1}^{\mu_{k+1}}Q_j(n_kx,n_ky),
\end{equation*}
where $n_k$ are properly defined integers, satisfies (\ref{v7}). The rest part of the proof is exactly the same as in the proof of Theorem \ref{T1},  so we leave it to our patient reader.
\end{proof}

\end{document}